\documentclass[a4paper, reqno]{amsart}

\title[{}]{Cosets from equivariant $\W$-algebras}

\author{Thomas Creutzig}
\address{Department of Mathematical and Statistical Sciences, University of Alberta, 632 CAB, Edmonton, Alberta, Canada T6G 2G1}
\email{creutzig@ualberta.ca}

\author{Shigenori Nakatsuka}
\address{Department of Mathematical and Statistical Sciences, University of Alberta, 632 CAB, Edmonton, Alberta, Canada T6G 2G1}
\email{shigenori.nakatsuka@ualberta.ca}

\usepackage{latexsym}
\usepackage{amsmath}
\usepackage{amsfonts}
\usepackage{amssymb}
\usepackage{amsxtra}
\usepackage{mathrsfs}
\usepackage{eucal}
\usepackage[all]{xy}
\usepackage{color}
\usepackage{braket}
\usepackage{graphics, setspace}
\usepackage{etoolbox, textcomp}
\usepackage{comment}
\usepackage{changepage}
\usepackage{xcolor}
\definecolor{rouge}{rgb}{0.85,0.1,.4}
\definecolor{bleu}{rgb}{0.1,0.2,0.9}
\definecolor{violet}{rgb}{0.7,0,0.8}
\usepackage[colorlinks=true,linkcolor=bleu,urlcolor=violet,citecolor=rouge]{hyperref}
\usepackage{cleveref}

\newtheorem{definition}{Definition}[section]
\newtheorem{proposition}[definition]{Proposition}
\newtheorem{theorem}[definition]{Theorem}
\newtheorem{corollary}[definition]{Corollary}
\newtheorem{lemma}[definition]{Lemma}

\numberwithin{equation}{section}

\newcommand{\Z}{\mathbb{Z}}

\newcommand{\C}{\mathbb{C}}

\newcommand{\Com}{\operatorname{Com}}

\newcommand{\K}{\mathbf{k}}
\newcommand{\A}{\mathbb{A}}

\newcommand{\Q}{\mathbb{Q}}
\newcommand{\F}{\mathbb{F}}
\newcommand{\g}{\mathfrak{g}}
\newcommand{\ga}{\mathfrak{a}}

\newcommand{\W}{\mathcal{W}}

\newcommand\doi[2]{\href{http://dx.doi.org/#1}{#2}}

\begin{document}
\maketitle

\begin{abstract}
The equivariant $\mathcal W$-algebra of a simple Lie algebra $\g$ is a BRST reduction of the algebra of chiral differential operators on the Lie group of $\g$. 
We construct a family of vertex algebras $A[\g, \kappa, n]$ as subalgebras of the equivariant $\mathcal W$-algebra of $\g$ tensored with the integrable affine vertex algebra $L_n(\check\g)$ of the Langlands dual Lie algebra $\check\g$ at level $n\in \Z_{>0}$. They are conformal extensions of  the tensor product of an affine vertex algebra and the principal $\W$-algebra whose levels satisfy a specific relation.

When $\g$ is of type $ADE$, we identify $A[\g, \kappa, 1]$ with the affine vertex algebra $V^{\kappa -1}(\g) \otimes L_1(\g)$, giving a new and efficient proof of the coset realization of the principal $\mathcal W$-algebras of type $ADE$. 
\end{abstract}

\section{Motivation and Results}\label{sec: motivation}

Let $G$ be a connected simply-connected simple algebraic group with Lie algebra $\g$. Then the space $\C[G]$ of regular functions decomposes into
\[
\C[G] \simeq \bigoplus_{\lambda \in P^+}  L_\lambda \otimes L_{\lambda^*}
\]
as a $\g \oplus \g$-module under the actions of the left and right invariant vector fields. 
Here $L_\lambda$ denotes the integrable $\g$-module of highest weight $\lambda$, $\lambda^*$ the highest weight of its dual representation, and $P^+$ the set of dominant integral weights.
This result can be chiralized by using the algebra of chiral differential operators over $G$ 
$$\mathcal{D}_{G, \kappa}^{\mathrm{ch}}=U(\hat{\g}_{\kappa})\otimes_{U(\g[\![t]\!])}\C[J_\infty G],$$
\cite{AG,GMS1,GMS2}. 
Here $J_\infty G$ is the jet scheme of $G$, on which $\g[\![t]\!]$ acts as the left invariant vector fields and $\hat{\g}_{\kappa}$ is the affine algebra of $\g$ at level $\kappa\in \C$.
This is a deformable family of vertex algebras whose top subspace is $\C[G]$.
The $\g$-action on $\C[G]$ via the right invariant vector fields is chiralized to an another action of $\hat{\g}_{\kappa^*}$ on $\mathcal{D}_{G, \kappa}^{\mathrm{ch}}$ at another specific level $\kappa^*$.
Let $V^\kappa(\g)$ denote the universal affine vertex algebra associated with $\g$ at level $\kappa\in \C$ and $\mathbb{V}^\kappa_\lambda$ the Weyl module of $V^\kappa(\g)$ whose top space is $L_\lambda$. Let $h^\vee$ denote the dual Coxeter number of $\g$, then for irrational levels $\kappa$ these two actions of the affine algebra decompose $\mathcal{D}_{G, \kappa}^{\mathrm{ch}}$ into
\begin{align}\nonumber
\mathcal{D}_{G, \kappa}^{\mathrm{ch}}\simeq \bigoplus_{\lambda\in P^+} \mathbb{V}^{\kappa}_\lambda\otimes \mathbb{V}^{\kappa^*}_{\lambda^*},\quad \frac{1}{\kappa+h^\vee}+\frac{1}{\kappa^*+h^\vee}=0.
\end{align}

There is a variant of $\mathcal{D}_{G, \kappa}^{\mathrm{ch}}$, called the quantum geometric Langlands kernel VOAs $\mathfrak{A}^{(n)}[\g,\kappa]$ \cite{CG}. These are families of vertex algebras, similar to $\mathcal{D}_{G, \kappa}^{\mathrm{ch}}$, but the inverses of the shifted levels of $\kappa, \kappa^*$ do not add up to zero but to integers in general. The existence of these algebras has been established by Moriwaki \cite{M} and the name has been partially justified as they serve as kernels for convolution operations mapping $\W$-algebras to the dual $\W$-superalgebras \cite{CLNS} as suggested by $S$-duality \cite[Section 7.3]{CG}.

A $\W$-algebra is a vertex algebra obtained from $V^\kappa(\g)$ via the BRST reduction $H_{\mathrm{DS},f}$, parametrized by nilpotent elements $f$ in $\g$ and in this work we are concerned with the principal $\W$-algebras, $\W^\kappa(\g)$, associated with the principal nilpotent element $f$ and we abbreviate $H_{\mathrm{DS},f}=H_{\mathrm{DS}}$ from now on.
The equivariant $\W$-algebra is realized as the BRST reduction of $\mathcal{D}_{G, \kappa}^{\mathrm{ch}}$ defined through the subalgebra $V^{\kappa}(\g)$.
By introducing $\W^{\kappa}(\g)$-modules $T^{\kappa}_{\lambda,0}:=H_{\mathrm{DS}}^0(\mathbb{V}^{\kappa }_{\lambda})$, we have for irrational levels $\kappa$
\begin{align}\nonumber
\mathcal{D}_{\kappa,G}^{W}\simeq \bigoplus_{\lambda\in P^+} T^{\kappa}_{\lambda,0}\otimes \mathbb{V}^{\kappa^*}_{\lambda^*}.
\end{align}

Two important results on $\W$-algebras have been established recently, motivated in part by $S$-duality in the physics.
The first one is the Arakawa--Frenkel duality \cite{AF}:
Arakawa and Frenkel introduced a variant of the BRST reductions where the differential is twisted by an automorphism associated to a coweight. 
They then proved isomorphisms between modules thus defined, generalizing the Feigin--Frenkel duality \cite{FF1} for algebras themselves, see Section \ref{sec:Walgebras}.
The second one is the Urod or translation property of the functor $H_{\mathrm{DS}}$ \cite{ACF}, that is, $H_{\mathrm{DS}}$ commutes with tensoring with the integrable representations of the affine algebra. 

These two results are the main ingredients in the proof of our first main theorem, which states the existence of the $\mathfrak{A}^{(n)}[\g,\kappa]$-analogue for $\mathcal{D}_{\kappa,G}^{W}$ with $G$ replaced by the Adjoint type $\mathrm{Ad}(G)$. We formulate it as a deformable family of vertex algebras depending on $\kappa$, i.e. a vertex algebra over a localization of the polynomial ring $\A=\C[\K]$ where $\K$ plays the role of the level $\kappa$, see Section \ref{Deformable family} for the precise definition.
Denote by $r^\vee$ the lacity of $\g$, by $Q$ its root lattice and set $Q^+ := Q \cap P^+$.
\begin{theorem}\label{thm: kernel object}
For $n \in \mathbb Z_{>0}$, there exists a deformable family of vertex algebras $A[\g, \kappa, n]$ over $\C$, such that at irrational levels $\kappa$
 $A[\g, \kappa, n]$ is simple and decomposes as a $V^{\kappa}(\g) \otimes \W^{\kappa^*}(\g)$-module into
\[
A[\g, \kappa, n] \simeq  \bigoplus_{\lambda \in Q^+} \mathbb{V}^{\kappa}_\lambda \otimes T^{\kappa^*}_{\lambda^*,0}.
\]
with $\kappa^*\in \C$ defined by the relation $ \frac{1}{\kappa+h^\vee} + \frac{1}{\kappa^*+h^\vee} = r^\vee n$.
\end{theorem}
\noindent
See Section \ref{proof1} for the explicit construction. In physics, vertex algebras appear at corners of topological boundary conditions \cite{GR, CG}.
For $\g = \mathfrak{sl}_m$ the algebra $A[\g, \kappa, n]$ is the $\mathcal D_{m, n}$-corner VOA discussed in Section 6.2.2 of \cite{CDGG}. Conjecturally a $\kappa \rightarrow \infty$ limit of these algebras are Feigin--Tipunin algebras \cite{FT} times a big center. It is desirable to study large level limits of $A[\g, \kappa, n]$ and in particular to settle this conjecture.
Note that deformable families of vertex algebras were introduced in order to study large level limits \cite{CL0, CL4}.

Let $L_\kappa(\g)$ and $L[\g, \kappa, n]$ denote the unique simple quotients of $V^\kappa(\g)$ and $A[\g, \kappa, n]$.
\begin{corollary}\label{cor} ${}$\\
(i) Let $\kappa$ be admissible and $\kappa^*$ non-degenerate (co)principal admissible, then
 $L[\g, \kappa, n]$ is a conformal extension of $L_\kappa(\g)\otimes \W_{\kappa^*}(\g)$ with
$\Com(L_\kappa(\g), L[\g, \kappa, n])\simeq \W_{\kappa^*}(\g)$. \\
(ii) $L[\g, \kappa, n]$ is strongly rational for $\kappa \in \mathbb Z_{>0}$.
\end{corollary}
\begin{proof}
The conformal weight $h(\lambda)$ of the top subspace of $\mathbb{V}^\kappa_\lambda \otimes T^{\kappa^*}_{\lambda^*,0}$ is 
\begin{align}\label{conformal weight in the decomposition}
h(\lambda) =  \frac{(\lambda, \lambda + 2\rho)}{2 (\kappa+h^\vee)} + \frac{(\lambda, \lambda + 2\rho)}{2 (\kappa^*+h^\vee)} - (\lambda, \rho^\vee) = \frac{(\lambda,\lambda) r^\vee n }{2} + (\lambda, n r^\vee \rho -\rho^\vee).
\end{align}
$\rho, \rho^\vee$ are the Weyl and dual Weyl vectors.
Hence, $h(\lambda) > 0$ for $\lambda \in Q^+\setminus \{ 0 \}$ and $L[\g,\kappa,n]$ is a $\frac{1}{2}\Z_{\geq0}$-graded vertex operator algebra.
Let $\widetilde V^\kappa(\g)$ and $\widetilde \W^{\kappa^*}(\g)$ be the images of $V^\kappa(\g)$ and $\W^{\kappa^*}(\g)$ in $L[\g, \kappa, n]$.
It follows from $\kappa\in \mathbb{R}_{\geq0}$ that $\text{Com}(\widetilde V^\kappa(\g), L[\g, \kappa, n]) = \widetilde \W^{\kappa^*}(\g)$ by \cite[Theorem 8.1]{CL0}.  
Then the simplicity of $L[\g,\kappa,n]$ implies $\widetilde \W^{\kappa^*}(\g)$ is simple, i.e $\W_{\kappa^*}(\g)$ by \cite[Theorem 4.1]{ACK}.
Since $\kappa\in \Z_{\geq0}$, $\kappa^*$ is non-degenerate principal or coprincipal admissible and thus $\W_{\kappa^*}(\g)$ is strongly rational \cite{A2, A3}.
Now, by \cite[Lemma 2.1]{ACKL}, $\text{Com}(\W_{\kappa^*}(\g), L[\g, \kappa, n])$ is simple, which implies $\widetilde V^\kappa(\g)=L_\kappa(\g)$ by \cite[Theorem 3.4]{AEM} since the coset is a conformal extension of $\widetilde V^\kappa(\g)$.
Therefore, $\text{Com}(\W_{\kappa^*}(\g), L[\g, \kappa, n])$ is strongly rational as a conformal extension of a strongly rational vertex operator algebra $L_\kappa(\g)$ of positive categorical dimension \cite[Theorem 1.1]{Mc} and so is $L[\g, \kappa, n]$ by \cite[Corollary 1.1]{CKM}
\end{proof}
In 1986, Goddard, Kent and Olive discussed how the Virasoro algebra, that is the $\W$-algebra of $\mathfrak{sl}_2$, is realized as a coset \cite{GKO}. 
The generalization to the principal $\W$-algebras of type $ADE$ is usually referred as the GKO-coset realization of $\W$-algebras. It has been widely used in physics, however it has only recently been proven \cite{ACL}. For type $A$ and $D$ it has then been reproven \cite{CL1, CL2}. 
A main motivation of this work is to give a much shorter proof of this famous theorem:
\begin{theorem}\label{thm: GKO}
Let $\g$ be of type $ADE$. Then $V^{\kappa -1}(\g) \otimes L_1(\g) \simeq A[\g, \kappa, 1]$ as vertex algebras over $\C$ for generic $\kappa\in \C\backslash \Q$.
In particular, $\mathrm{Com}(V^\kappa(\g), V^{\kappa -1}(\g) \otimes L_1(\g)) \simeq \W^{\kappa^*}(\g)$ holds where $\kappa^*$ is defined by $ \frac{1}{\kappa+h^\vee} + \frac{1}{\kappa^*+h^\vee} = 1$.
\end{theorem}
\noindent
See Section \ref{sec: Proof of Theorem 2.3} for the proof. The proof says that we have a map $V^{\kappa -1}_S(\g) \otimes_\C L_1(\g) \rightarrow A[\g, \kappa, 1]_S$ as deformable families over an \'{e}tale cover $\mathrm{Spec}S$ of a Zariski open subset of $\C$. 
The isomorphism $\text{Com}(L_\kappa(\g), L_{\kappa -1}(\g) \otimes L_1(\g)) \simeq \W_{\kappa^*}(\g)$ when the level $\kappa-1$ is admissible follows from the same argument as in the proof of Corollary \ref{cor}.

A variant of the method for $\g$ of type $C$ shows that $V^\kappa(\mathfrak{osp}_{1|2n})$ is an extension of $V^\kappa(\mathfrak{sp}_{2n})$ and $\W^{\kappa^*}(\mathfrak{sp}_{2n})$ with  $ \frac{1}{\kappa+h^\vee} + \frac{1}{\kappa^*+h^\vee} = 2$, which is key for understanding ordinary modules of $L_\kappa(\mathfrak{osp}_{1|2n})$ in \cite{CGL}.

\vspace{2mm}
\noindent{\bf Acknowledgements} The work of T.C. is supported by NSERC Grant Number RES0048511 and the work of S.N is supported by JSPS Overseas Research
Fellowships Grant Number 202260077.

\section{Integral form of the equivariant $\W$-algebra}\label{sec: integral form}
We follow the notation in Section \ref{sec: motivation}. In particular, we denote by $\g$ a simple Lie algebra and by $Q\subset P$ the root and weight lattice, respectively. Switching to the Langlands dual Lie algebra $\check{\g}$ of $\g$, we use the symbol $\check{X}$ for $\check{\g}$ corresponding to $X$ for $\g$, e.g. $\check{Q}$ stands for the root lattice of $\check{\g}$.
\subsection{Deformable family of vertex algebras}\label{Deformable family}
Let $\mathbb{A}=\C[\mathbf{k}]$ denote the polynomial ring in the variable $\K$ and $\F=\C(\K)$ the field of rational functions. 
Given a subset $U\subset \C$, consider a family of vertex algebras (or its modules) $V^\kappa$ depending on the parameter $\kappa\in \C\backslash U$.
We say that it is a deformable family of vertex algebras \cite{CL0} if there exists a vertex algebra $V_{\A_U}$ over $\A_U:=\A[\frac{1}{\K-a}\mid a\in U]$ which is $\A_U$-free and satisfies $V_{\A_U}\otimes_{\A_U}\C_\kappa\simeq V^\kappa$ ($\kappa\in \C\backslash U$).
We often write $V^\kappa_{\A_U}$ to keep the parameter $\kappa$ explicit.
It is useful to view $V_\F^\kappa:=V_{\A_U}^\kappa\otimes_{\A_U}\F$ as the vertex algebra capturing the behavior of $V^\kappa$ for generic $\kappa$. When $U$ is empty, i.e. $\A_U=\A$, we call $V^\kappa_\A$ the integral form of $V^\kappa$. 
Note that the affine vertex algebra $V^\kappa(\g)$ ($\kappa\in \C$) has an integral form $V_\A^\kappa(\g)=U_{\A}(\widehat{\g}_{\A})\otimes_{U(\g[\![t]\!]}\C$ where $\widehat{\g}_\A$ is the affine Lie algebra over $\A$ whose level $\kappa$ is replaced by $\K$ and $U_{\A}(\widehat{\g}_{\A})$ is its enveloping algebra over $\A$. Similarly, we have the integral forms of Weyl modules $\mathbb{V}_{\lambda,\A}^\kappa:=U_{\A}(\widehat{\g}_{\A})\otimes_{U(\g[\![t]\!])}L_\lambda$.

\subsection{Principal $\W$-algebra}\label{sec:Walgebras}
Let us recall the BRST reduction functor $H_{\mathrm{DS}}$. 
For this, let $\mathfrak{h}\subset \g$ denote the Cartan subalgebra, $\g= \mathfrak{n}_+\oplus \mathfrak{h}\oplus \mathfrak{n}_-$ the triangular decomposition,  and $\mathfrak{n}_+=\oplus_{\alpha\in \Delta_+}\g_\alpha$ the decomposition into root subspaces with $\Delta_+$ the set of positive roots. We denote by $\Pi\subset \Delta_+$ the set of simple roots.  We fix (non-zero) root vectors $e_\alpha\in \g_\alpha$ with structure constants $c_{\alpha,\beta}^\gamma$. 
Then given a $V^\kappa(\g)$-module $M$, the BRST reduction $H_{\mathrm{DS}}(M)$ with coefficients in $M$ is, by definition, the cohomology of the complex 
$$C(M)=M\otimes \bigwedge{}^{\frac{\infty}{2}+\bullet}(\mathfrak{n}_+)$$
equipped with the differential $d=d_{\mathrm{st}}+d_{\chi}$ given by 
\begin{align*}
&d_{\mathrm{st}}=\int dz \sum_{\alpha\in \Delta_+}e_{\alpha}(z)\otimes \psi_{\alpha_i}^*(z)-\frac{1}{2}\sum_{\alpha,\beta,\gamma\in \Delta_+}c_{\alpha,\beta}^\gamma:\psi_\gamma(z)\psi_\alpha^*(z)\psi_\beta^*(z):, \\
&d_{\chi}=\int dz \sum_{\alpha\in \Pi}\psi^*_\alpha(z)=\sum_{\alpha\in \Pi}\psi^*_{\alpha,1}.
\end{align*}
Here $\bigwedge{}^{\frac{\infty}{2}+\bullet}(\mathfrak{n}_+)$ is a tensor product of $bc$-systems generated by the fields $\psi_{\alpha}(z)$, $\psi_{\alpha}^*(z)$, ($\alpha\in \Delta_+$), satisfying the OPE $\psi_{\alpha}(z)\psi_{\beta}^*(w)\sim \delta_{\alpha,\beta}/(z-w)$.
In particular, the vertex algebra $\W^\kappa(\g)=H_{\mathrm{DS}}^0(V^\kappa(\g))$ is called the principal $\W$-algebra and enjoys the Feigin--Frenkel duality \cite{FF1} 
\begin{align}\label{FF duality}
\W^\kappa(\g)\simeq \W^{\check{\kappa}}(\check{\g}),\quad r^\vee(\kappa+h^\vee)(\check{\kappa}+\check{h}^\vee)=1.
\end{align}
Note that the functor $H_{\mathrm{DS}}$ is well-defined over $\A$. 
By \cite[Appendix B]{BFN}, $\W_\A^\kappa(\g):=H_{\mathrm{DS}}^0(V_\A(\g))$ gives the integral form of $\W^\kappa(\g)$ and \eqref{FF duality} is refined to 
$\W_\F^\kappa(\g) \simeq \W_\F^{\check{\kappa}}(\check{\g})$ with $r^\vee(\K+h^\vee)(\check{\K}+\check{h}^\vee)=1$, see \cite{ACL, BFN} for more on the duality for integral forms.
We also have the integral forms of $H_{\mathrm{DS}}(\mathbb{V}_{\lambda}^\kappa)$:
\begin{proposition}\label{cohomology vanishing}\hspace{0mm}\\
(i) $H_{\mathrm{DS}}^n(\mathbb{V}_{\lambda,\A})=0$ if $n\neq 0$.\\
(ii) $H_{\mathrm{DS}}^0(\mathbb{V}_{\lambda,\A})$ is $\A$-free and satisfies $H_{\mathrm{DS}}^0(\mathbb{V}_{\lambda,\A})\otimes_\A \C_\kappa \simeq H_{\mathrm{DS}}^0(\mathbb{V}_\lambda^\kappa)$ for $\kappa\in \C$.
\end{proposition}
\proof
The case $\lambda=0$, i.e.\ $V_{\lambda,\A}=V_\A(\g)$ is proven in \cite[Appendix B]{BFN} by upgrading the proof in \cite[Chapter 15]{FBZ} for the case over $\C$ to the case over $\A$. As the proof in \cite[Chapter 15]{FBZ} is straightforwardly generalized to the case $\lambda \in P^+$ \cite[Section 4.2]{AF}, the assertion is proven by word-by-word translation of the argument in \cite[Appendix B]{BFN} to the case $\lambda \in P^+$, following \cite[Section 4.2]{AF}.
\endproof
To introduce more $\W^\kappa(\g)$-modules, we use the twisted BRST reduction $H_{\mathrm{DS},\check{\mu}}$ ($\check{\mu}\in \check{P}^+$). Given a $V^\kappa(\g)$-module $M$, we introduce a $C(V^\kappa(\g))$-module $\sigma_{\check{\mu}}^*(C(M))$ on the vector space $C(M)$ by using the Li's $\Delta$-operator $\Delta(-\check{\mu}_\Delta,z)$ with
$$\check{\mu}_\Delta(z)=\check{\mu}(z)+\sum_{\alpha\in \Delta_+}(\check{\mu},\alpha):\psi_\alpha(z)\psi^*_\alpha(z):.$$
More explicity, we set $\Delta(-\check{\mu}_\Delta,z)=z^{-\check{\mu}_{\Delta,0}}\mathrm{exp}(\sum_{n=1}^\infty\frac{\check{\mu}_{\Delta,n}}{n}(-z)^{-n})$ and 
\begin{align}\label{twisted action}
C(V^\kappa(\g))\rightarrow \mathrm{End}(C(M))[\![z^{\pm1}]\!],\quad A\mapsto Y_{C(M)}\left(\Delta(-\check{\mu}_\Delta,z)A,z\right).
\end{align}
The twisted BRST reduction $H_{\mathrm{DS},\check{\mu}}$ is defined as the cohomology $H_{\mathrm{DS},\check{\mu}}(M)=H_{\mathrm{DS}}(\sigma_{\check{\mu}}^*(M))$, which is naturally a module over $\W^\kappa(\g)$ by \eqref{twisted action}.
Equivalently, we modify the differential $d$ as $d_{\check{\mu}}=d_{\mathrm{st}}+d_{\chi,\check{\mu}}$ with
$$d_{\chi,\check{\mu}}=\sum_{\alpha\in \Pi}\psi^*_{\alpha,(\check{\mu},\alpha)+1}.$$

\begin{proposition}[\cite{AF}]\label{twistedDS} Let $\lambda\in P^+$, $\check{\mu}\in\check{P}^+$ and set  $T^{\kappa}_{\lambda,\check{\mu}}=H^0_{\mathrm{DS},\check{\mu}}(\mathbb{V}^\kappa_\lambda)$.\\
\textup{(i)} For $\kappa\in \C$, $H_{\mathrm{DS},\check{\mu}}^n(\mathbb{V}^\kappa_\lambda)=0$ if $n\neq0$.\\
\textup{(ii)} For $\kappa\in \C\backslash \Q$, there is an isomorphism 
$$T^{\kappa}_{\lambda,\check{\mu}} \simeq \check{T}^{\check{\kappa}}_{\check{\mu},\lambda}$$
of modules over $\W^\kappa(\g)\simeq \W^{\check{\kappa}}(\check{\g})$ \eqref{FF duality}.
\end{proposition}
\noindent
We note that the proposition also holds over $\F$.

\subsection{Equivariant $\W$-algebra}
We construct the integral form of the equivariant vertex algebra $\mathcal{D}_{G,\A}^W$ over $\A$ for $\mathcal{D}_{G,\kappa}^W$ by using the following integral form of the algebra of chiral differential operators 
$$\mathcal{D}_{G,\A}^{\mathrm{ch}}=U_{\A}(\widehat{\g}_{\A})\otimes_{U(\g[\![t]\!])}\C[J_\infty G].$$
\begin{theorem}\label{integral form of equiv W}\hspace{0mm}\\
(i) $H_{\mathrm{DS}}^n(\mathcal{D}_{G,\A}^{\mathrm{ch}})=0$ if $n\neq 0$.\\
(ii) $H_{\mathrm{DS}}^0(\mathcal{D}_{G,\A}^{\mathrm{ch}})$ is $\A$-free and satisfies 
$H_{\mathrm{DS}}^0(\mathcal{D}_{G,\A}^{\mathrm{ch}})\otimes_{\A}\C_\kappa \simeq \mathcal{D}_{G,\kappa}^W$.
\end{theorem}
\noindent
Therefore, $H_{\mathrm{DS}}^0(\mathcal{D}_{G,\A}^{\mathrm{ch}})$ is the integral form $\mathcal{D}_{G,\A}^W$ of $\mathcal{D}_{G,\kappa}^W$.

The vertex algebra $\mathcal{D}_{G,\A}^{\mathrm{ch}}$ has a conformal vector $\omega(z)$, \cite{GMS1}. After the base change $\mathcal{D}_{G,\A}^{\mathrm{ch}}\otimes_{\A}\A[\frac{1}{\K+h^\vee}]$, it is the sum of the Sugawara vectors of affine vertex subalgebras $V^\kappa_{\A}(\g^L)$ and $V^{-\kappa-2h^\vee}_{\A}(\g^R)$ over $\A$ corresponding to the left and right  invariant vector fields
Since the PBW base theorem implies 
$$\mathcal{D}_{G,\A}^{\mathrm{ch}}\simeq \A \otimes_\C U(\g[t^{-1}]t^{-1})\otimes_\C \C[J_\infty G]$$
as $\A$-modules, each homogeneous subspace of the conformal weight decomposition 
$$\mathcal{D}_{G,\A}^{\mathrm{ch}}=\bigoplus_{\Delta\geq0}\mathcal{D}_{G,\A}^{\mathrm{ch}}[\Delta]$$
is $\A$-free and seen as the base change $\A\otimes_\C M[\Delta]$ of some semisimple $(\g^L,\g^R)$-bimodule $M[\Delta]\subset U(\g[t^{-1}]t^{-1})\otimes_\C \C[J_\infty G]$ over $\C$. In particular, we have
$$M[0]=\C[G] \simeq \bigoplus_{\lambda\in P^+} L_\lambda\otimes_\C L_{\lambda^*}.$$
It follows that $M[\Delta]$ ($\Delta\geq0$) decoompses into
\begin{align}\label{decomposition of each conformal weight}
M[\Delta]\simeq \bigoplus_{\lambda\in P^+} M^\lambda[\Delta]\otimes L_{\lambda^*}
\end{align}
as a $(\g^L,\g^R)$-bimodule where $M^\lambda[\Delta]$ is a finite dimensional semisimple $\g^L$-module consisting of the highest weight vectors of weight $\lambda$ for the $\g^R$-action. 
Now, let $\widehat{M}_\A[\leq \Delta]$ denote the $(V_\A^L(\g),\g^R)$-sub-bimodule of $\mathcal{D}_{G,\A}^{\mathrm{ch}}$ generated by $M[p]$ with $p\leq\Delta$.
In partucular, we have 
\begin{align}\label{starting point}
\widehat{M}_\A[\leq 0]\simeq \bigoplus_{\lambda\in P^+}\mathbb{V}_{\lambda,\A}\otimes_\C L_{\lambda^*}\subset \mathcal{D}_{G,\A}^{\mathrm{ch}}.
\end{align}
Here $\mathbb{V}_{\lambda,\A}=U_\A(\widehat{\g}_\A)\otimes_{U(\g[\![t]\!])}L_\lambda $ is the integral form of the Weyl module of highest weight $\lambda$.
By using the decomposition \eqref{decomposition of each conformal weight}, we find that $\widehat{M}_\A[\leq \Delta]$ decomposes into
$$\widehat{M}_\A[\leq \Delta]\simeq \bigoplus_{\lambda} \widehat{M}^\lambda_\A[\leq \Delta]\otimes_\C L_{\lambda^*} $$ 
where $\widehat{M}^\lambda_\A[\leq \Delta]$ is a finite successive extension of $\mathbb{V}_{\lambda,\A}$ by Weyl modules $\mathbb{V}_{\mu,\A}$.

In order to apply Proposition \ref{cohomology vanishing} to $\mathcal{D}_{G,\A}^{\mathrm{ch}}$, we use an inductive argument which is verified by the following easy lemma:
\begin{lemma}\label{induction}
Let $0\rightarrow M_1 \rightarrow M_2 \rightarrow M_3\rightarrow 0$ be an exact sequence of $V_{\A}(\g)$-modules such that the following property holds for $i=1,3$:\\
\textup{(P1)} $H_{\mathrm{DS}}^n (M_i)=0$ ($n\neq 0$) and $H_{\mathrm{DS}}^0 (M_i)$ is $\A$-free.\\
\textup{(P2)} $H_{\mathrm{DS}}^0 (M_i)\otimes_\A \C_\kappa\simeq H_{\mathrm{DS}}^0 (M_i\otimes_\A \C_\kappa)$.\\
Then \textup{(P1)-(P2)} also hold for $i=2$.
\end{lemma}
\proof
(P1) is immediate from the long exact sequence for the functor $H_{\mathrm{DS}}$. Then the K$\ddot{\text{u}}$nneth spectral sequence implies that $H_{\mathrm{DS}}^0$ commutes with the base change and thus \textup{(P2)}.
\endproof

\begin{proof}[Proof of Theorem \ref{integral form of equiv W}]
It follows from \eqref{starting point} and Proposition \ref{cohomology vanishing} that 
\begin{itemize}
\item $H_{\mathrm{DS}}^n(\widehat{M}_\A[\leq p])=0$ if $n\neq 0$.
\item $H_{\mathrm{DS}}^0(\widehat{M}_\A[\leq p])$ is $\A$-free and $H_{\mathrm{DS}}^0(\widehat{M}_\A[\leq p])\otimes_\A \C_\kappa\simeq H_{\mathrm{DS}}^0(\widehat{M}_\A[\leq p]\otimes_\A \C_\kappa)$.
\end{itemize}
hold for $p=0$. 
Then the case for general $p$ follows by induction (Lemma \ref{induction}) since $\widehat{M}_\A[\leq p]$ is a finite successive extension of $\widehat{M}_\A[\leq 0]$ by Weyl modules. The assertion hold since the direct limit commutes with base change -$\otimes_\A \C_\kappa$ and taking cohomology
$H_{\mathrm{DS}}^n(\mathcal{D}_{G,\A}^W)\simeq \varinjlim H_{\mathrm{DS}}^n(\widehat{M}_\A[\leq p])$.
\end{proof}

\section{Proof of Theorem \ref{thm: kernel object}}\label{proof1}
We first work over the field $\F$. The argument will be completely the same as working over $\C$ for irrational levels, which will verify the assertion in Theorem \ref{thm: kernel object} for irrational levels.
 
Let $L_n(\check{\g})$ be the simple affine vertex algebra (over $\C$) with $n\in \Z_{\geq0}$. Then we have a decomposition 
\begin{align}\label{nGKO}
V^{\check{\kappa}}_\F(\check{\g})\otimes_\C L_n(\check{\g})\simeq \bigoplus_{\check{\mu}\in \check{Q}^+} \mathbb{V}^{\check{\kappa}+n}_{\check{\mu},\F}\otimes_\F C_{\check{\mu},\F}
\end{align}
as a module over the diagonal $V^{\check{\kappa}+n}_\F(\check{\g})$-action. Here $C_{\check{\mu},\F}$ is the multiplicity space consisting of highest weight vectors of weight $\check{\mu}$ at level $\check{\kappa}+n$, which is a simple module over the coset $\mathrm{Com}(V^{\check{\kappa}+n}_\F(\check{\g}),V^{\check{\kappa}}_\F(\check{\g})\otimes_\C L_n(\check{\g}))$ by  \cite[Theorem 4.12]{CL1}.
For $\lambda\in P$, we have by \cite[Theorem 7.1]{ACF} (which is stated over $\C$, but also holds over $\F$ ) an isomorphism of modules 
\begin{align*}
H^0_{\mathrm{DS},\lambda}(V^{\check{\kappa}}_\F(\check{\g})\otimes_\C L_n(\check{\g}))\simeq H^0_{\mathrm{DS},\lambda}(V^{\check{\kappa}}_\F(\check{\g}))\otimes_\C (\sigma_{\lambda}^*L_n(\check{\g}))
\end{align*}
over $H^0_{\mathrm{DS}}(V^{\check{\kappa}}_\F(\check{\g})\otimes_\C L_n(\check{\g}))\simeq H^0_{\mathrm{DS}}(V^{\check{\kappa}}_\F(\check{\g}))\otimes_\C L_n(\check{\g})$. Here $\sigma_{\lambda}^*$ is the twist defined similarly to \eqref{twisted action}.
Then by setting $\lambda\in Q^+$, \eqref{nGKO} implies 
\begin{equation}
\begin{split}
&H_{\mathrm{DS},\lambda}^0(V^{\check{\kappa}}_\F(\check{\g}))\otimes_\C L_n(\check{\g})
\simeq H_{\mathrm{DS},\lambda}^0(V^{\check{\kappa}}_\F(\check{\g})\otimes_\C L_n(\check{\g}))\\
&\hspace{1cm}\simeq\bigoplus_{\check{\mu}\in \check{Q}^+} H_{\mathrm{DS},\lambda}^0(\mathbb{V}^{\check{\kappa}+n}_{\check{\mu},\F})\otimes_\F C_{\check{\mu},\F}
\simeq \bigoplus_{\check{\mu}\in \check{Q}^+} \check{T}_{\check{\mu},\lambda,\F}^{\check{\kappa}+n}\otimes_\F C_{\check{\mu},\F}.
\end{split}
\end{equation}
Combining it with Proposition \ref{twistedDS} (ii), we find 
\begin{align}\label{rel of T}
T^\kappa_{\lambda,0,\F}\otimes_\C L_n(\check{\g})\simeq \bigoplus_{\check{\mu}\in \check{Q}^+}T^{\varkappa}_{\lambda,\check{\mu},\F}\otimes_\F C_{\check{\mu},\F}
\end{align}
where the levels $\kappa$, $\varkappa$ are defined by the relations 
\begin{align}\label{gluing rel}
r^\vee (\kappa+h^\vee) (\check{\kappa}+\check{h}^\vee)=1,\quad r^\vee (\check{\kappa}+n+\check{h}^\vee)(\varkappa+h^\vee)=1.
\end{align}

Next, we invoke the equvariant $\W$-algebra $\mathcal{D}_{\mathrm{Ad}(G),\A}^W$ associated with the algebraic group $\mathrm{Ad}(G)$ of Adjoint type. 
It is the fixed point of $\mathcal{D}_{G,\A}^W$ for the center $Z(G)=\mathrm{Ker}(G\twoheadrightarrow \mathrm{Ad}(G))$ and thus its base change $\mathcal{D}_{\mathrm{Ad}(G),\F}^W$ decomposes into 
\begin{align}\label{eqCDO}
\mathcal{D}_{\mathrm{Ad}(G),\F}^W\simeq \bigoplus_{\lambda\in Q^+} T^{\kappa}_{\lambda,0,\F}\otimes_\F \mathbb{V}^{\kappa^*}_{\lambda^*,\F},\quad \frac{1}{\kappa+h^\vee}+\frac{1}{\kappa^*+h^\vee}=0
\end{align}
as a module over $\W^{\kappa}_\F(\g)\otimes_\F V^{\kappa^*}_\F(\g)$. Then it follows from \eqref{rel of T} and \eqref{eqCDO} that 
\begin{align}\label{generic decomposition}
\mathcal{D}_{\mathrm{Ad}(G),\F}^W\otimes_\F L_n(\check{\g})
&\nonumber \simeq \bigoplus_{\lambda\in Q^+}  T_{\lambda,0,\F}^\kappa \otimes_\F \mathbb{V}^{\kappa^*}_{\lambda^*,\F} \otimes_\C L_n(\check{\g})\\
&\simeq \bigoplus_{\begin{subarray}c\lambda\in Q^+\\ \check{\mu}\in \check{Q}^+\end{subarray}} \mathbb{V}^{\kappa^*}_{\lambda^*,\F} \otimes_\F T^{\varkappa}_{\lambda,\check{\mu},\F}\otimes_\F C_{\check{\mu},\F}.
\end{align}
Now, introduce the following vertex algebra over $\A$
$$A[\g,\kappa^*,n]_{\A}:=\left(\mathcal{D}_{\mathrm{Ad}(G),\A}^W\otimes_\C L_n(\check{\g})\right) \cap \mathrm{Com}\left(C_{0,\F}, \mathcal{D}_{\mathrm{Ad}(G),\F}^W\otimes_\C L_n(\check{\g}) \right).$$
We show that it is a desired deformable family of vertex algebras. 
Since $\mathcal{D}_{\mathrm{Ad}(G),\A}^W\otimes_\C L_n(\check{\g})$ is $\A$-free by Theorem \ref{integral form of equiv W} and $\A$ is a P.I.D., the $\A$-submodule $A[\g,\kappa^*,n]_{\A}$ is $\A$-free by \cite[Theorem 5.1]{HS}.
Consider the natural map 
$$A[\g,\kappa^*,n]_{\A} \otimes_\A \C_\kappa\rightarrow (\mathcal{D}_{\mathrm{Ad}(G),\A}^W\otimes_\C L_n(\check{\g}))\otimes_\A \C_\kappa,\quad (\kappa\in \C\backslash \Q).$$
It is immediate from the definition of $A[\g,\kappa^*,n]_{\A}$ that the map is injective.
Since \eqref{generic decomposition} also holds for irrational levels as mentioned already, we obtain
\begin{align}\label{branching rule at irrational levels}
A[\g,\kappa^*,n]_{\A} \otimes_\A \C_\kappa\hookrightarrow \bigoplus_{\lambda\in Q^+} \mathbb{V}^{\kappa^*}_{\lambda^*} \otimes_\C T^\varkappa_{\lambda,0}.
\end{align}
Since $A[\g,\kappa^*,n]_{\A}$ is $\A$-free, the character remains the same under specialization, which agrees with the character of the RHS.
Therefore, \eqref{branching rule at irrational levels} is an isomorphism. 
Finally, since $\mathcal{D}_{\kappa,\mathrm{Ad}(G)}^W$ is simple \cite[Section 6]{A}, the decomposition \eqref{generic decomposition} implies that $A[\g,\kappa^*,n]$ is also simple for $\kappa^*\in \C\backslash \Q$ by \cite[Proposition 5.4]{CGN}. 
Finally, it follows from \eqref{gluing rel} and \eqref{eqCDO} that the levels $\kappa^*$ and $\varkappa$ satisfy the relation $\frac{1}{\kappa^*+h^\vee}+\frac{1}{\varkappa+h^\vee}=r^\vee n$. This completes the proof of Theorem \ref{thm: kernel object}.

\section{Proof of Theorem \ref{thm: GKO}}\label{sec: Proof of Theorem 2.3}
Suppose that $\g$ is of type $ADE$ and consider the case $n=1$ in Theorem \ref{thm: kernel object}.
Let us calculate the character of $A[\g,\kappa,1]$ for $(\mathfrak{h},L_0)$ with $L(z)=\sum_{m\in\Z}L_mz^{-m-2}$ the standard Virasoro field.
By \cite[Eq.\ (5.7)]{AF}, we have 
$$
\mathrm{ch}[T^{\kappa^*}_{\lambda^*,0}]=  \frac{q^{\frac{(\lambda^*+2\rho,\lambda)}{2(\kappa^*+h^\vee)}+(\rho,\rho)}}{(q;q)_\infty^{\mathrm{rank}\g}} \sum_{w\in W}\varepsilon(w)q^{-(w(\lambda^*+\rho),\rho)} 
$$
where $W$ is the Weyl group and $(a_1,\cdots,a_m;q)_\infty=\prod_{1\leq i\leq m, p\geq0}(1-a_iq^m)$. 
Also 
$$
\mathrm{ch}[\mathbb{V}^\kappa_{\lambda}] = \frac{1}{D}\sum_{\lambda\in Q^+}q^{\frac{(\lambda+2\rho,\lambda)}{2(\kappa+h^\vee)}}\mathrm{ch}[L_{\lambda}], 
\qquad D = (q;q)_\infty^{\mathrm{rank}\g}\prod_{\alpha\in \Delta_+}(e^\alpha q,e^{-\alpha} q;q)_\infty
$$
Then
\begin{equation}
\begin{split}
 \mathrm{ch}[A[\g,\kappa,1&]] =\sum_{\lambda\in Q^+}\mathrm{ch}[\mathbb{V}^\kappa_{\lambda}]\mathrm{ch}[T^{\kappa^*}_{\lambda^*,0}]\\
&=  \frac{1}{D}\sum_{\lambda\in Q^+}q^{\frac{(\lambda+2\rho,\lambda)}{2(\kappa+h^\vee)}}\mathrm{ch}[L_{\lambda}]\frac{q^{\frac{(\lambda^*+2\rho,\lambda^*)}{2(\kappa^*+h^\vee)}+(\rho,\rho)}}{(q;q)_\infty^{\mathrm{rank}\g}}  \sum_{w\in W}\varepsilon(w)q^{-(w(\lambda^*+\rho),\rho)}\\
& = \frac{1}{D}\frac{1}{\displaystyle  (q;q)_\infty^{\mathrm{rank}\g}  }\sum_{\lambda\in Q^+}q^{\frac{(\lambda,\lambda)}{2}}\mathrm{ch}[L_\lambda]\sum_{w\in W}\varepsilon(w)q^{(\lambda+\rho-w(\lambda+\rho),\rho)}\\
&\label{character}=\frac{1}{D }\frac{1}{(q;q)_\infty^{\mathrm{rank}\g}}\sum_{\lambda\in Q}q^{\frac{(\lambda,\lambda)}{2}}e^\lambda 
= \mathrm{ch}[V^{\kappa-1}(\g)] \, \mathrm{ch}[L_1(\g)]
\end{split}
\end{equation} 
by \cite[Theorem 4.1]{KW}. From the last equality, we find that $\mathrm{ch}[A[\g,\kappa,1]_\F]$ agrees with the character of $V^{\kappa-1}_\F(\g)\otimes_\C L_1(\g)$ for $(\mathfrak{h},L_0)$ with the diagonal $\mathfrak{h}$-action. Let $A[\g,\kappa,1]_\F=\oplus_{\Delta\geq0}A[\g,\kappa,1]_{\F,\Delta}$ denote the decomposition by the $L_0$-action. Then \eqref{character} implies that the subspaces of conformal weight one and zero admits the integral form and are isomorphic to 
\begin{align}\label{lower cwt subspaces}
A[\g,\kappa,1]_{\A,1}\simeq \g_\A\oplus \text{ad}(\g)_\A,\quad A[\g,\kappa,1]_{\A,0}\simeq \A
\end{align}
as modules over $\g_\A(:=\A\otimes_\C\g)$ corresponding to $V^\kappa_\A(\g) \subset A[\g,\kappa,1]_\A$. Here $\mathrm{ad}(\g)$ is the adjoint representation of $\g$. 
Set $T(\g) = \g[t]/(t^2)$ the Takiff algebra of $\g$. 
\begin{lemma}\label{ALemma}
Let $\g$ be a simple Lie algebra, then $\mathrm{Hom}_{\g_R}(\wedge^2 \mathrm{ad}(\g)_R, \mathrm{ad}(\g)_R) \simeq R$ for any $\C$-algebra $R$. 
\end{lemma}
\begin{proof}
It suffices to show $\mathrm{Hom}_\g(\wedge^2 \mathrm{ad}(\g), \mathrm{ad}(\g)) \simeq \mathbb C$.
Let $V$ be the vertex superalgebra of free fermions in the adjoint representation, so that $V_1 \simeq \wedge^2 \text{ad}(\g)$ as $\g$-modules for the conformal weight one subspace $V_1$.
By \cite[Chapter 3.7]{F}, there is a conformal embedding of  $L_{h^\vee}(\g)\hookrightarrow V$ and thus $V_1 \simeq \text{ad}(\g) \oplus M$ for an integrable $\g$-module $M$. Then $M$ must be the top subspace of an $L_{h^\vee}(\g)$-module of conformal weight  $C_M / (4h^\vee) =1$, with $C_M$ the Casimir eigenvalue of $M$, i.e. $C_M = 4h^\vee$. 
Since $C_{\text{ad}(\g)} = 2h^\vee$, $\mathrm{ad}(\g)$ does not appear in $M$ as a direct summand. 
\end{proof}
\begin{lemma}\label{BLemma}
Given an integral domain $R$, let $\g_R \subset \ga$ be Lie algebras over $R$ such that $\g$ is as above and $\ga \simeq  \g_R \oplus \mathrm{ad}(\g)_R$ as $\g_R$-modules. Then, $\ga_S$ is isomorphic to $\g_S \oplus \g_S$ or $T(\g)_S$ as Lie algebras (over $S$) for some finitely generated $R$-algebra $S$.
\end{lemma}
\begin{proof}
Fix a decomposition $\ga=\g_{1,R}\oplus \g_{2,R}$ where $\g_{i,R}$ is a copy of $\g_R$ so that 
$[x_1,y_1]=[x,y]_1$, $[x_1,y_2]=[x,y]_2$, ($x,y\in \g_R$). Then by Lemma \ref{ALemma}, there exist $\alpha,\beta\in R$ such that 
$[x_2,y_2]=\alpha[x,y]_1+\beta[x,y]_2.$
(i) $4\alpha+\beta^2=0$: one can show that the map
$$\varphi\colon \g_R\rightarrow \ga,\quad x\mapsto -\beta/2 x_1+x_2$$
is an embedding of $\g_{1,R}$-modules such that $[\varphi(x),\varphi(y)]=0$. Therefore, $\ga\simeq T(\g)_R$ as Lie algebras in this case. 
(ii) $4\alpha+\beta^2\neq0$: set $S:=R(\sqrt{4\alpha+\beta^2})$. One can show that the maps 
$$\varphi_\pm\colon \g_S\rightarrow \ga_S,\quad x\mapsto (p_\pm x_1+x_2)/(2p_\pm+\beta),\quad p_\pm=(-\beta\pm \sqrt{4\alpha+\beta^2})/2,$$
are Lie algebra homomorphisms over $S$ and satisfy 
$$[x_1,\varphi_\pm(y)]=\varphi_\pm([x,y]),\quad [x_2,\varphi_\pm(y)]=(p_\pm+\beta)\varphi_\pm([x,y]),$$
i.e. $\mathrm{Im}\varphi_\pm$ are ideals.
Thus $(\varphi_+,\varphi_-)\colon \g_S\oplus \g_S \simeq \ga_S$ as Lie algebras over $S$.
\end{proof}
Let us consider the space of symmetric invariant bilinear forms on $T(\g)$, i.e. $\mathbf{B}(T(\g))=\mathrm{Hom}_{T(\g)}(\mathrm{Sym}^2(T(\g)),\C)$. 
Since $\mathbf{B}(\g)=\C \kappa_0$ with $\kappa_0$ the standard form on $\g$, $\mathrm{Hom}_{\g}(\mathrm{Sym}^2(T(\g)),\C)$ is spanned by 
$\kappa_{a}$, $\kappa_{b}$, $\kappa_{c}$, which are $\kappa_0$ on the direct summands $\g_1\otimes \g_1$, $\g_1\otimes \g_2$, $\g_2\otimes \g_2$, respectively, where $T(\g)=\g\oplus \g t=\g_1\oplus \g_2$. By the $T(\g)$-invariance, we find $\mathbf{B}(T(\g))=\C\kappa_{a}\oplus \C\kappa_{b}$.
Note that $k_a\kappa_{a} + k_b \kappa_{b}$ is non-degenerate if and only if $k_b \neq 0$ and that we have an isomorphism of affine vertex algebras 
$V^{k_a\kappa_{a} + k_b \kappa_{b}}(T(\g))\simeq V^{k_a(\kappa_{a} + \kappa_{b})}(T(\g))$ for $k_ak_b\neq0$ induced by the automorpisms $\C^*$ of $T(\g)$ via the scalings on $\g t$. Note that this argument also folds over field extending $\C$, in particular $\F$.
\begin{corollary}\label{ACor}
There is an isomorphism $A[\g,\kappa,1]_{S,1}\simeq \g_S\oplus \g_S$ of Lie algebras where $S$ is taken as in Lemma \ref{BLemma}.
\end{corollary}
\begin{proof}
By Lemma \ref{BLemma}, $A[\g,\kappa,1]_{S,1}$ is isomorphic to $\g_S\oplus \g_S$ or $T(\g)_S$. 
Suppose $A[\g,\kappa,1]_{S,1}\simeq T(\g)_S$, in which case $A[\g,\kappa,1]_{\A,1}\simeq T(\g)_\A$ holds by the proof of Lemma \ref{BLemma}.
Then we have a non-zero homomorphism $V^{\kappa \kappa_{a}+k_b \kappa_{b}}_\A(T(\g))\rightarrow  A[\g,\kappa,1]_\A$ for some $k_b\in \A$.
Since $A[\g,\kappa,1]_\F$ is simple, by \eqref{lower cwt subspaces}, $A[\g,\kappa,1]_\F$ is conic and thus it admits a non-degenerate pairing \cite{Li} that descends to a non-degenerate symmetric invariant bilinear form on $A[\g,\kappa,1]_{\A,1}\simeq T(\g)_\A$, i.e. $k_b \neq 0$.
It follows that $V^{\kappa\kappa_{a} + k_b \kappa_{b}}_{\A^\circ}(T(\g))\simeq V^{\kappa(\kappa_a +  \kappa_b)}_{\A^\circ}(T(\g))$ with $\A^\circ=\A[\frac{1}{k_b}]$, whose specialization at generic $\kappa$ is simple by \cite[Theorem 3.6 (1) with $f=0$]{CL1}. Hence $V^{\kappa(\kappa_{a} +  \kappa_{b})}(T(\g))$ is a subalgebra of $A[\g,\kappa,1]$ for such $\kappa$, a contradiction since
$\mathrm{ch}[V^{\kappa(\kappa_{a} +  \kappa_{b})}(T(\g))] = \mathrm{ch}[V^\kappa(\g)]^2 \not \leq \mathrm{ch}[V^{\kappa-1}(\g)\otimes L_1(\g)] = \mathrm{ch}[A[\g,\kappa,1]]$.
\end{proof}
\proof[Proof of Theorem \ref{thm: GKO}]
By Corollary \ref{ACor}, we have a homomorphism
\begin{align}\label{homomorphism}
V^{\kappa_1}_S(\g)\otimes_S V^{\kappa_2}_S(\g)\rightarrow A[\g,\kappa,1]_S,\quad \kappa_1+\kappa_2=\K,
\end{align}
of vertex algebras over $S$ for some $\kappa_i\in S$. Here $S$ is a finitely generated integral domain over $\A$.
We claim that either $\kappa_1 = 1$ or $\kappa_2 =1$. 
Since $L_1(\g)$ has a non-zero singular vector of conformal weight two and of weight $2\theta$ where $\theta$ is the highest root of $\g$. 
It follows from \eqref{character} that the same is true for $A[\g,\kappa,1]_S$ and thus for $V^{\kappa_1}_S(\g)\otimes_S V^{\kappa_2}_S(\g)$ by \eqref{homomorphism}.
Let $e^a(z), e^b(z)$ (resp. $f^a(z), f^b(z)$) denote the fields corresponding to the highest (resp. lowest) root vector for the first and second factor.
Then the singular vector inside $V^{\kappa_1}_S(\g)\otimes_S V^{\kappa_2}_S(\g)$ must be of the form
$$
x =  \alpha  e_{-1}^a e_{-1}^a  |0\rangle  +  \beta e_{-1}^a e_{-1}^b  |0\rangle   +\gamma  e_{-1}^b e_{-1}^b  |0\rangle
$$
for some $\alpha, \beta, \gamma\in S$. Since $x$ satisfies $f^c_1f^d_1 x = 0$ for $c, d \in \{ a, b \}$, we have
\begin{itemize}
\item $f_1^af_1^ax=0$ implies either $\alpha =0$ or $\kappa_1 \in \{0, 1\}$.
\item  $f_1^bf_1^bx=0$ implies either $\gamma =0$ or $ \kappa_2 \in \{ 0, 1\}$.
\item $f_1^af_1^b x=0$ implies either $\beta =0$ or $\kappa_1\kappa_2 = 0$.
\end{itemize}
Since $V^0_S(\g)$ has a singular vector at conformal weight one, but not at weight two, we have $\kappa_1,\kappa_2\neq0$, which implies either $\kappa_1 =1$ or $\kappa_2 =1$ holds. Therefore, \eqref{homomorphism} factors through 
\begin{align*}
g_S\colon V^{\kappa-1}_S(\g)\otimes_\C L_1(\g)\rightarrow A[\g,\kappa,1]_S.
\end{align*}
By specializing at generic $\kappa\in \C\backslash \Q$, we obtain an isomorphism  
$$V^{\kappa-1}(\g)\otimes_\C L_1(\g)\xrightarrow{\simeq} A[\g,\kappa,1]$$
by the simplicity of $V^{\kappa-1}(\g)\otimes_\C L_1(\g)$ and the coincidence of the characters \eqref{character}.
This completes the proof.
\endproof

\end{document}